\documentclass{amsart}
\usepackage{amscd}
\usepackage{thmdefs}
\usepackage{amsmath}
\usepackage{amssymb}
\usepackage{textcomp}
\newtheorem{cs}{Case}

\def\Lip{\mathrm{Lip}}
\def\Lipo{\mathrm{Lip}_0}

\def\lip{\mathrm{lip}}
\def\lipo{\mathrm{lip}_0}

\def\N{\mathbb{N}}
\def\K{\mathbb{K}}

\begin{document}
\title[Weakly compact composition operators on spaces of Lipschitz functions]{Weakly compact composition operators \\ on spaces of Lipschitz functions}
\author{A. Jim\'{e}nez-Vargas}
\address{Departamento de \'{A}lgebra y An\'{a}lisis Matem\'{a}tico, Universidad de Almer\'{i}a, 04120 Almer\'{i}a, Spain}
\email{ajimenez@ual.es}
\date{\today}


\begin{abstract}
Let $X$ be a pointed compact metric space. Assuming that $\lipo(X)$ has the uniform separation property, we prove that every weakly compact composition operator on spaces of Lipschitz functions $\Lipo(X)$ and $\lipo(X)$ is compact.
\end{abstract}
\maketitle

\section{Introduction}

A \emph{composition operator} $C_\phi$ on a function space $F(X)$ over a set $X$ is a linear operator from $F(X)$ into itself defined by $C_\phi(f)=f\circ\phi$, where $\phi$ is a map from $X$ into $X$ called the \emph{symbol} of $C_\phi$. Boundedness and compactness of operators $C_\phi$ have been intensively studied in terms of the properties of $\phi$ for different function spaces. See the monograph by Singh and Manhas \cite{sm} and the references therein for a comprehensive treatment of this subject. 

Our aim in this paper is to study weakly compact composition operators on spaces of Lipschitz functions. Let $X$ and $Y$ be metric spaces. We use the letter $d$ to denote the distance in any metric space. A map $f\colon X\to Y$ is said to be \emph{Lipschitz} if 
$$ 
\sup\left\{\frac{d(f(x),f(y))}{d(x,y)}\colon x,y\in X,\, x\neq y\right\}<\infty ;
$$
and \emph{supercontractive} if 
$$
\lim_{d(x,y)\to 0}\frac{d(f(x),f(y))}{d(x,y)}=0,
$$
meaning that the following property holds:
$$
\forall \varepsilon>0,\; \exists \delta>0\colon x,y\in X,\, 0<d(x,y)<\delta \Rightarrow \frac{d(f(x),f(y))}{d(x,y)}<\varepsilon .
$$
Constant maps are Lipschitz and supercontractive but, for example, the identity function on $\mathbb{R}$ is Lipschitz, but not supercontractive; whereas the function $n\mapsto n^2$ on $\mathbb{N}$ is supercontractive, but not Lipschitz. A supercontractive Lipschitz function is often called a \emph{little Lipschitz function}.

Let $X$ be a pointed compact metric space with a base point which we always will denote by $0$, and let $\mathbb{K}$ be the field of real or complex numbers. The \emph{Lipschitz space} $\Lipo(X)$ is the Banach space of all Lipschitz functions $f\colon X\to\K$ for which $f(0)=0$, endowed with the Lipschitz norm
$$
\Lip_d(f)=\sup\left\{\frac{\left|f(x)-f(y)\right|}{d(x,y)}\colon x,y\in X, \; x\neq y\right\},
$$
and the \emph{little Lipschitz space} $\lipo(X)$ is the closed subspace of $\Lipo(X)$ formed by all little Lipschitz functions. These spaces have been largely investigated along the time. We refer the reader to Weaver's book \cite{wea} for a complete study on them.

There are $\lipo$ spaces containing only the zero function as, for instance, $\lipo[0,1]$ with the usual metric, but there exist also some large classes of $\lip_0$ spaces which separate points, even uniformly, in the sense introduced by Weaver in \cite[Definition 3.2.1]{wea} as follows. Given a pointed compact metric space $X$, it is said that $\lipo(X)$ \emph{separates points uniformly} if there exists a constant $a>1$ such that, for every $x,y\in X$, there exists $f\in\lipo(X)$ with $\Lip_d(f)\leq a$ such that $f(x)=d(x,y)$ and $f(y)=0$. This happens, for example, when $X$ is \emph{uniformly discrete}, meaning that $d(x,y)\geq \delta$ for all $x,y\in X$, $x\neq y$, for some $\delta>0$; or when $X$ is the middle-thirds Cantor set with the metric inherited from $[0,1]$. Also, $\lip_0(X^\alpha)$ has the uniform separation property, where $X^\alpha=(X,d^\alpha)$ for some $0<\alpha <1$ (see \cite[Proposition 3.2.2]{wea}). Lipschitz functions on $X^\alpha$ are called \emph{H\"{o}lder functions of exponent} $\alpha$. It is worth to point out that for any pointed compact metric space $X$, there exists a pointed compact metric space $Y$ such that $\lip_0(Y)$ has the uniform separation property and $\lipo(X)$ is isometrically isomorphic to $\lip_0(Y)$ (see \cite[Corollary 4.4.9]{wea}).

The composition operators on spaces of Lipschitz functions have been studied by some authors. Weaver characterized in \cite[Proposition 1.8.2]{wea} the boundedness of composition operators $C_\phi$ on $\Lipo(X)$ by means of the Lipschitz condition of their symbols, when $X$ is a pointed complete metric space. The completeness on $X$ is not restrictive in view of \cite[Proposition 1.2.3]{wea}. Assuming that $X$ is a compact metric space, Kamowitz and Scheinberg \cite{ks} gave a complete description of compact composition operators $C_\phi$ on both Banach spaces $\Lip(X)$ of scalar-valued Lipschitz functions on $X$ with the norm $\max\{\Lip_d,\left\|\cdot\right\|_\infty\}$ and Banach spaces $\lip(X,d^\alpha)$ $(0<\alpha <1)$ of scalar-valued little Lipschitz functions on $X^\alpha$ with the norm $\max\{\Lip_{d^\alpha},\left\|\cdot\right\|_\infty\}$, in terms of the supercontractive property of their symbols. For pointed metric spaces $X$, not necessarily compact, this characterization was extended in \cite{jv13} to both spaces $\lipo(X)$ satisfying the uniform separation property on bounded subsets of $X$ and spaces $\Lipo(X)$. In \cite{bj11}, Botelho and Jamison provided a characterization of compact weighted composition operators on spaces of vector valued Lipschitz functions. 

Our aim in this paper is to prove that every weakly compact composition operator $C_\phi$ on both $\lipo(X)$ and $\Lipo(X)$ is compact provided that $\lipo(X)$ has the uniform separation property. The key tool to prove this result is the fact, stated in \cite[Corollary 3.3.5]{wea}, that $\lipo(X)$ has the uniform separation property if and only if $\lipo(X)^{**}$ is isometrically isomorphic to $\Lipo(X)$.

\section{The results}

We prepare the proof of our main result by stating first a new characterization of supercontractive functions. 
Given a metric space $X$, for $x\in X$ and $r>0$, we denote by $B(x,r)$ the open ball in $X$ of center $x$ and radius $r$. See \cite{jllv13} for an analogous characterization of Lipschitz functions.  

\begin{lemma}\label{lem1}
Let $X$ be a pointed compact metric space and let $\phi\colon X\to X$ be a continuous map. If $\phi$ is not supercontractive, then there exist a real number $\varepsilon>0$, two sequences $\{x_n\}$ and $\{y_n\}$ in $X$ converging to a point $x_0\in X$ such that $0<d(x_n,y_n)<1/n$ and $\varepsilon<d(\phi(x_n),\phi(y_n))/d(x_n,y_n)$ for all $n\in\N$, and a function $f\in\Lipo(X)$ such that $f(\phi(x_n))=d(\phi(x_n),\phi(y_n))$ and $f(\phi(y_n))=0$ for all $n\in\N$.
\end{lemma}

\begin{proof}
Since $\phi$ is not supercontractive, we can take a real number $\varepsilon>0$ and two sequences $\{p_n\}$ and $\{q_n\}$ in $X$ such that $0<d(p_n,q_n)<1/n$ and $\varepsilon<d(\phi(p_n),\phi(q_n))/d(p_n,q_n)$ for all $n\in\N$. Notice that $\phi(p_n)\neq \phi(q_n)$ for all $n\in\N$. Since $X$ is compact, passing to a subsequence if necessary, we may suppose that $\{p_n\}$ converges to $x_0\in X$. Since $d(q_n,x_0)\leq d(q_n,p_n)+d(p_n,x_0)\leq 1/n+d(p_n,x_0)$ for all $n\in\N$, $\{q_n\}$ also converges to $x_0$.

We will construct two sequences $\{x_n\}$ and $\{y_n\}$ in $X$ converging to $x_0$ such that 
$$
0<d(x_n,y_n)<\frac{1}{n},\qquad \varepsilon<\frac{d(\phi(x_n),\phi(y_n))}{d(x_n,y_n)},\qquad d(\phi(y_n),\phi(x_0))\leq d(\phi(x_n),\phi(x_0))
$$
for all $n\in\N$. In addition, we will prove that there exists a sequence of pairwise disjoint open balls $\{B(\phi(x_n),r_n)\}$ such that $\phi(y_n)\notin\bigcup_{j=1}^{\infty}B(\phi(x_j),r_j)$ for all $n\in\N$. To this end, we consider the sets 
$$
A=\left\{n\in\N \colon\phi(p_n)=\phi(x_0)\right\},\qquad B=\left\{n\in\N \colon\phi(q_n)=\phi(x_0)\right\}
$$
and distinguish two cases.

\begin{cs} Suppose that $A$ or $B$ are infinite. If $A$ is infinite, then there exists a strictly increasing map $\sigma\colon\N\to\N$ such that $\phi(p_{\sigma(n)})=\phi(x_0)$ for all $n\in\N$. Notice that $\phi(q_{\sigma(n)})\neq\phi(p_{\sigma(n)})=\phi(x_0)$ for each $n\in\N$ and $\phi(q_{\sigma(n)})\to\phi(x_0)$. Hence there is a subsequence $\{q_{\sigma(\upsilon(n))}\}$ such that $d(\phi(q_{\sigma(\upsilon(n+1))}),\phi(x_0))<(1/3)d(\phi(q_{\sigma(\upsilon(n))}),\phi(x_0))$ for all $n\in\N$. For each $n\in\N$, define $x_n=q_{\sigma(\upsilon(n))}$ and $y_n=p_{\sigma(\upsilon(n))}$. Then we have 
$$
0<d(x_n,y_n)<\frac{1}{\sigma(\upsilon(n))}\leq\frac{1}{n},\qquad \varepsilon<\frac{d(\phi(x_n),\phi(y_n))}{d(x_n,y_n)},\qquad d(\phi(y_n),\phi(x_0))=0<d(\phi(x_n),\phi(x_0)).
$$
Moreover, $d(\phi(x_{n+1}),\phi(x_0))<(1/3)d(\phi(x_n),\phi(x_0))$ for all $n\in \N$. Set 
$$
r_n=\frac{1}{2}\min\left\{d(\phi(x_n),\phi(x_0)),d(\phi(x_n),\phi(y_n))\right\}
$$
for each $n\in\N$. As $\phi(y_n)=\phi(x_0)$ for all $n\in \N$, it follows that $r_n=d(\phi(x_n),\phi(x_0))/2$. Note that if $n<m$, then $r_m<r_n/3$ and $d(x,\phi(x_0))<3r_m<r_n$ for any $x\in B(\phi(x_m),r_m)$. This implies that, for each $n\in\N$ and any $m>n$, we have $B(\phi(x_m),r_m)\subset B(\phi(x_0),r_n)$. As $B(\phi(x_n),r_n)\cap B(\phi(x_0),r_n)=\varnothing$ for all $n\in\N$, we conclude that the balls $B(\phi(x_n),r_n)$ are pairwise disjoint and $\phi(y_n)=\phi(x_0)\not\in\bigcup_{j=1}^{\infty}B(\phi(x_j),r_j)$ for all $n\in\N$. Therefore $\{x_n\}$ and $\{y_n\}$ satisfy the required conditions. The same argument works if $B$ is infinite.
\end{cs}

\begin{cs}
Suppose that $A$ and $B$ are both finite. Let $r=\max(A\cup B)$. Note that $\phi(p_{n+r})\neq\phi(x_0)$ and $\phi(q_{n+r})\neq\phi(x_0)$ for all $n\in \N$. Define the sequences $\{t_n\}$ and $\{s_n\}$ by
\begin{align*}
t_n&=\left\{
      \begin{array}{ll}
        p_{n+r} & \text{ if } d(\phi(q_{n+r}),\phi(x_0))\leq d(\phi(p_{n+r}),\phi(x_0)), \\
        q_{n+r} & \text{ if } d(\phi(p_{n+r}),\phi(x_0))< d(\phi(q_{n+r}),\phi(x_0)),
      \end{array}
    \right.\\
s_n&=\left\{
      \begin{array}{ll}
        q_{n+r} & \text{ if } d(\phi(q_{n+r}),\phi(x_0))\leq d(\phi(p_{n+r}),\phi(x_0)), \\
       p_{n+r} & \text{ if } d(\phi(p_{n+r}),\phi(x_0))< d(\phi(q_{n+r}),\phi(x_0)).
      \end{array}
    \right.
\end{align*}
Note that $d(\phi(s_n),\phi(x_0))\leq d(\phi(t_n),\phi(x_0))$ for all $n\in\N$. As $\{t_n\}$ converges to $x_0$, take a subsequence $\{t_{\sigma(n)}\}$ for which 
$$
d(\phi(t_{\sigma(n+1)}),\phi(x_0))<\frac{1}{3}\min\{d(\phi(s_{\sigma(n)}),\phi(x_0)),d(\phi(t_{\sigma(n)}),\phi(s_{\sigma(n)})) \}
$$
for all $n\in\N$. Let $x_n=t_{\sigma(n)}$ and $y_n=s_{\sigma(n)}$. Then  
$$
0<d(x_n,y_n)<\frac{1}{\sigma(n)+r}<\frac{1}{\sigma(n)}\leq\frac{1}{n},\qquad \varepsilon<\frac{d(\phi(x_n),\phi(y_n))}{d(x_n,y_n)},\qquad d(\phi(y_n),\phi(x_0))\leq d(\phi(x_n),\phi(x_0))
$$
for all $n\in\N$. Moreover, a straightforward induction yields that, for each $n\in\N$ and any $m>n$, 
$$
d(\phi(x_m),\phi(x_0))<\frac{1}{3}\min\left\{d(\phi(y_n),\phi(x_0)),d(\phi(x_n),\phi(y_n))\right\}.
$$
For each $n\in\N$, take  
$$
r_n=\frac{1}{2}\min\left\{d(\phi(x_n),\phi(x_0)),d(\phi(x_n),\phi(y_n))\right\}.
$$ 
Fix $n,m\in\N$ such that $m>n$. As $d(\phi(x_m),\phi(x_0))<d(\phi(y_n),\phi(x_0))/3\leq d(\phi(x_n),\phi(x_0))/3$ and $d(\phi(x_m),\phi(x_0))<d(\phi(x_n),\phi(y_n))/3$, we have $d(\phi(x_m),\phi(x_0))<2r_n/3$. Also, $r_m\leq d(\phi(x_m),\phi(x_0))/2<r_n/3$ and it is easy to check that $B(\phi(x_m),r_m)\subset B(\phi(x_0),r_n)$. Since $B(\phi(x_n),r_n)\cap B(\phi(x_0),r_n)=\varnothing$, it follows that $B(\phi(x_n),r_n)\cap B(\phi(x_m),r_m)=\varnothing$. Moreover, as $d(\phi(y_m),\phi(x_0))\leq d(\phi(x_m),\phi(x_0))<2r_n/3$, it is clear that $\phi(y_m)\notin B(\phi(x_n),r_n)$. Finally, from the inequalities
\begin{align*}
r_m & \leq \frac{d(\phi(x_m),\phi(x_0))}{2}<\frac{d(\phi(y_n),\phi(x_0))}{6}<d(\phi(y_n),\phi(x_0))-\frac{d(\phi(y_n),\phi(x_0))}{3}\\
& < d(\phi(y_n),\phi(x_0))-d(\phi(x_m),\phi(x_0))\leq d(\phi(y_n),\phi(x_m)),
\end{align*}
we infer that $\phi(y_n)\not\in B(\phi(x_m),r_m)$. Then we can conclude that the balls $B(\phi(x_n),r_n)$ are pairwise disjoint and $\phi(y_n) \not\in \bigcup_{j=1}^\infty B(\phi(x_j),r_j)$ for all $n\in\N$.
\end{cs}

Finally, we prove that there exists a function $f\in\Lipo(X)$ such that $f(\phi(x_n))=d(\phi(x_n),\phi(y_n))$ and $f(\phi(y_n))=0$ for all $n\in\N$. Indeed, for each $n$, let 
$$
h_n(x)=\max\left\{0,1-\frac{d(x,\phi(x_n))}{r_n}\right\}.
$$
Notice that $h_n$ is Lipschitz with $\Lip_d(h_n)\leq 1/r_n$, $h_n(\phi(x_n))=1$ and $h_n(x)=0$ for all $x\in X\setminus B(\phi(x_n),r_n)$ (see \cite[Lemma 2.1]{VLV09}). Define $g\colon X\to\mathbb{R}$ by
$$
g(x)=\sum_{n=1}^{\infty} d(\phi(x_n),\phi(y_n))h_n(x).
$$
Observe that $g(x)=0$ for any $x\notin\bigcup_{j=1}^\infty B(\phi(x_j),r_j)$ and so $g(\phi(y_n))=0$ for all $n\in\N$. As the balls $B(\phi(x_n),r_n)$ are pairwise disjoint, if $x\in\bigcup_{j=1}^\infty B(\phi(x_j),r_j)$, then $g(x)=d(\phi(x_m),\phi(y_m))h_m(x)$ for some fixed $m \in \N$ (depending only on $x$) and, in particular, $g(\phi(x_n))=d(\phi(x_n),\phi(y_n))$ for all $n\in\N$. Moreover, as $d(\phi(x_n),\phi(y_n))\leq 2 d(\phi(x_n),\phi(x))$, we have that $d(\phi(x_n),\phi(y_n))\leq 4r_n$, hence it must be that $\Lip_d(d(\phi(x_n),\phi(y_n))h_n)\leq 4$ and so $g$ is Lipschitz. 

Finally, if $\phi(x_0)=0$, then $g(\phi(y_n))\to g(0)$ by the continuity of $\phi$ and $g$, but since $g(\phi(y_n))=0$ for all $n$, it follows that $g(0)=0$ and so $g\in\Lipo(X)$. Hence we can take $f=g$ and the lemma follows. Otherwise, if $\phi(x_0)\neq 0$, take $\varepsilon=d(\phi(x_0),0)/2>0$. Since $\{\phi(x_n)\}$ converges to $\phi(x_0)$, there exists $m\in\N$ such that $\varepsilon\leq d(\phi(x_{n+m}),0)$ for all $n\in\N$. Then the sequences $\{x_{n+m}\}$ and $\{y_{n+m}\}$ and the function $f\colon X\to\mathbb{R}$, defined by 
$$
f(x)=\left(1 -\max\left\{0,1-\frac{d(x,0)}{\varepsilon}\right\}\right)g(x),
$$
satisfy the required conditions in the lemma.
\end{proof}


A second tool that we will use to prove the main result is the fact that the Arens product on $\lipo(X)^{**}$ coincides with the pointwise product on $\lipo(X)^*$ whenever $\lipo(X)$ separates points uniformly (see \cite[Theorem 3.8]{bcd} for the case $\lip(X^\alpha)$ with $0<\alpha<1$).

Let $A$ be a commutative Banach algebra. The \emph{Arens product} on $A^{**}$ is defined in stages as follows (see \cite{arens}). For any $a,b\in A$, $f\in A^*$ and $F,G\in A^{**}$, define 
$$
(f\diamond a)(b)= f(ab),\qquad (F\diamond f)(a)= F(f\diamond a),\qquad (F\diamond G)(f)= F(G\diamond f).
$$
Then $A^{**}$ is a Banach algebra under this product and it is denoted by $(A^{**},\diamond)$. The algebra $A$ is said to be \emph{Arens regular} if the algebra $(A^{**},\diamond)$ is commutative. 

By \cite[Theorems 3.3.3 and 2.2.2]{wea}, the space $\lipo(X)^{**}$ is isometrically isomorphic to $\Lipo(X)$ provided that $\lipo(X)$ separates points uniformly, under the isometric isomorphism $\Phi\colon\lipo(X)^{**}\to\Lipo(X)$ defined by 
\begin{equation}\label{patri}
\Phi(F)(x)=F(\delta_x)\qquad (F\in\lipo(X)^{**}, \ x\in X).
\end{equation}
In fact, this identification is the most natural since 
\begin{equation}\label{alex}
\Phi(Q_X(f))(x)=(Q_X(f))(\delta_x)=\delta_x(f)=f(x) \quad (f\in\lipo(X), \ x\in X),
\end{equation}
where $Q_X$ denotes the canonical injection from $\lipo(X)$ into $\lipo(X)^{**}$. We now can prove the following. 
 
\begin{lemma}\label{lem3}
Let $X$ be a pointed compact metric space and assume that $\lipo(X)$ separates points uniformly. Then the Arens product on $\lipo(X)^{**}$ coincides with the pointwise product on $\lipo(X)^*$, and so $\lipo(X)$ is Arens regular.
\end{lemma}

\begin{proof}
Let $x\in X$, $f,g\in\lipo(X)$ and $F,G\in\lipo(X)^{**}$. First, we have
$$
(\delta_x\diamond f)(g)=\delta_x(fg)=f(x)g(x)=f(x)\delta_x(g),
$$
and so $\delta_x\diamond f=f(x)\delta_x$. Now, using (\ref{patri}), we compute 
$$
(G\diamond\delta_x)(f)
=G(\delta_x\diamond f)
=G(f(x)\delta_x)
=f(x)G(\delta_x)
=G(\delta_x)\delta_x(f),
$$
and so $G\diamond\delta_x=G(\delta_x)\delta_x$. Finally, we have
$$
(F\diamond G)(\delta_x)
=F(G\diamond\delta_x)
=F(G(\delta_x)\delta_x)
=F(\delta_x)G(\delta_x)
=(FG)(\delta_x).
$$
Since the closed linear hull of the set $\{\delta_x\colon x\in X\}$ in $\Lipo(X)^*$ coincides with $\lipo(X)^{*}$ by \cite[Theorem 3.3.3]{wea}, we infer that $(F\diamond G)(\gamma)=(FG)(\gamma)$ for all $\gamma\in\lipo(X)^*$, and so $F\diamond G=FG$, as desired.
\end{proof}

We now are ready to prove the main result of this paper.

\begin{theorem}\label{teo1}
Let $X$ be a pointed compact metric space and let $\phi\colon X\to X$ be a base point-preserving Lipschitz map. Suppose that $\lipo(X)$ separates points uniformly. Then every weakly compact operator $C_{\phi}\colon\lipo(X)\to\lipo(X)$ is compact.
\end{theorem}

\begin{proof}
By Gantmacher's Theorem, the operator $C_{\phi}\colon \lipo(X)\to\lipo(X)$ is weakly compact if and only if $C_{\phi}^{**}(\lipo(X)^{**})$ is contained in $Q_X(\lipo(X))$. 

We claim that this inclusion means that $C_{\phi}(\Lipo(X))$ is contained in $\lipo(X)$. To prove this, consider $\lipo(X)^{**}$ as a Banach algebra with the Arens product $\diamond$. Since $C_{\phi}$ is an algebra endomorphism of $\lipo(X)$, then $C_{\phi}^{**}$ is an algebra endomorphism of $(\lipo(X)^{**},\diamond)$. Indeed, for any $F,G\in\lipo(X)^{**}$ and $\gamma\in\lipo(X)^*$, it is clear that 
$C_{\phi}^{**}(F)(\gamma)=(F\circ C_{\phi}^*)(\gamma)=F(\gamma\circ C_{\phi})$, and, by applying Lemma \ref{lem3}, we have 
\begin{gather*}
C_{\phi}^{**}(F\diamond G)(\gamma)
=(F\diamond G)(\gamma\circ C_{\phi})
=(FG)(\gamma\circ C_{\phi})
=F(\gamma\circ C_{\phi})G(\gamma\circ C_{\phi})\\
=C_{\phi}^{**}(F)(\gamma)C_{\phi}^{**}(G)(\gamma)
=(C_{\phi}^{**}(F)C_{\phi}^{**}(G))(\gamma)
=(C_{\phi}^{**}(F)\diamond C_{\phi}^{**}(G))(\gamma).
\end{gather*}
Note also that $\Phi$ is an algebra homomorphism from $(\lipo(X)^{**},\diamond)$ to $\Lipo(X)$ since 
$$
\Phi(F\diamond G)(x)
=(F\diamond G)(\delta_x)
=(FG)(\delta_x)
=F(\delta_x)G(\delta_x)
=\Phi(F)(x)\Phi(G)(x)
=(\Phi(F)\Phi(G))(x)
$$
for all $x\in X$. It follows that $\Phi C_{\phi}^{**}\Phi^{-1}$ is an algebra endomorphism of $\lipo(X)$. By \cite[Corollary 4.5.6]{wea}, $\Phi C_{\phi}^{**}\Phi^{-1}$ is of the form $C_{\varphi}$ for some base point-preserving Lipschitz map $\varphi\colon X\to X$. Using the equalities (\ref{alex}) and $Q_X C_{\phi}=C_{\phi}^{**}Q_X$, we infer that  
$$
C_{\phi}(f)=\Phi(Q_X(C_{\phi}(f)))=\Phi(C_{\phi}^{**}(Q_X(f)))=C_{\varphi}(\Phi(Q_X(f)))=C_{\varphi}(f)
$$
for all $f\in\lipo(X)$. Hence $C_{\phi}=C_{\varphi}$ on $\lipo(X)$, which implies $\phi=\varphi$ since $\lipo(X)$ separates points. It follows that 
$$
C_{\phi}^{**}(\lipo(X)^{**})\subset Q_X(\lipo(X)) 
\Leftrightarrow 
\Phi C_{\phi}^{**}\Phi^{-1}(\Lipo(X))\subset 
\lipo(X) 
\Leftrightarrow C_{\phi}(\Lipo(X))\subset\lipo(X),
$$
and this proves our claim. Hence $C_{\phi}\colon \lipo(X)\to \lipo(X)$ is weakly compact if and only if $f\circ\phi\in\lipo(X)$ for all $f\in\Lipo(X)$. 

Assume now that $C_{\phi}\colon\lipo(X)\to \lipo(X)$ is not compact. Then $\phi$ is not supercontractive by the version of \cite[Theorem 1.3]{jv13} for $\lipo$ spaces (see \cite{ks} for the case $\lip(X^\alpha)$ with $0<\alpha<1$). By applying Lemma \ref{lem1}, there exist a real number $\varepsilon>0$, two sequences $\{x_n\}$ and $\{y_n\}$ in $X$ such that $0<d(x_n,y_n)<1/n$ and $\varepsilon<d(\phi(x_n),\phi(y_n))/d(x_n,y_n)$ for all $n\in\N$ and a function $f\in\Lipo(X)$ such that $f(\phi(x_n))=d(\phi(x_n),\phi(y_n))$ and $f(\phi(y_n))=0$ for all $n\in\N$. Then we have 
$$
\frac{\left|(f\circ\phi)(x_n)-(f\circ\phi)(y_n))\right|}{d(x_n,y_n)}=\frac{d(\phi(x_n),\phi(y_n))}{d(x_n,y_n)}>\varepsilon
$$
for all $n\in\N$. Since $d(x_n,y_n)\to 0$ as $n\to\infty$, we infer that $f\circ\phi$ is not in $\lipo(X)$ and thus $C_{\phi}\colon \lipo(X)\to \lipo(X)$ is not weakly compact. This completes the proof of the theorem.
\end{proof}


From Theorem \ref{teo1} we derive the following.

\begin{corollary}\label{cor1}
Let $X$ be a pointed compact metric space and let $\phi\colon X\to X$ be a base point-preserving Lipschitz map. Suppose that $\lipo(X)$ separates points uniformly. Then every weakly compact operator $C_{\phi}\colon\Lipo(X)\to\Lipo(X)$ is compact.
\end{corollary}

\begin{proof}
Assume that $C_{\phi}\colon\Lipo(X)\to\Lipo(X)$ is weakly compact. Hence $\Phi^{-1} C_{\phi}\Phi\colon \lipo(X)^{**}\to\lipo(X)^{**}$ is weakly compact. Given $F\in\lipo(X)^{**}$ and $x\in X$, we have
\begin{gather*}
(\Phi^{-1}C_{\phi}\Phi)(F)(\delta_x)=((C_{\phi}\Phi)(F))(x)=C_{\phi}(\Phi(F))(x)=\Phi(F)(\phi(x)) \\
=F(\delta_{\phi(x)})=F(\delta_x\circ C_{\phi})=(F\circ C^*_\phi)(\delta_x)=C^{**}_\phi(F)(\delta_x).
\end{gather*}
Since linear combinations of elements of the form $\delta_x$ are dense in $\lipo(X)^{*}$, 
it follows that $\Phi^{-1} C_{\phi}\Phi=C^{**}_\phi$. Hence $C^{**}_\phi\colon \lipo(X)^{**}\to\lipo(X)^{**}$ is weakly compact and so also is $C_\phi\colon \lipo(X)\to\lipo(X)$. Then, by Theorem \ref{teo1}, $C_\phi\colon \lipo(X)\to\lipo(X)$ is compact. Hence $\phi$ is supercontractive by the version of \cite[Theorem 1.3]{jv13} for $\lipo$ spaces, and finally $C_\phi\colon \Lipo(X)\to\Lipo(X)$ is compact by \cite[Theorem 1.2]{jv13}.
\end{proof}

\begin{remark}
We must point out that analogous versions of the preceding results can be established for composition operators $C_\phi$ from $\lipo(X)$ to $\lipo(Y)$ and from $\Lipo(X)$ to $\Lipo(Y)$. Moreover, every space $\Lip(X)$ is isometrically isomorphic to a certain space $\Lipo(X_0)$ by \cite[Theorem 1.7.2]{wea}, and the same isometric isomorphism identifies also the spaces $\lip(X)$ and $\lipo(X_0)$ (see \cite[p. 74]{wea}). Using these identifications, it is easy to show that Theorem \ref{teo1} and Corollary \ref{cor1} hold also for composition operators from $\lip(X)$ to $\lip(Y)$ and from $\Lip(X)$ to $\Lip(Y)$.
\end{remark}

\textbf{Acknowledgments.} The author thanks Nik Weaver for his valuable help in a first draft of this paper. 

\bibliographystyle{amsplain}

\end{document}